\documentclass[12pt]{amsart}
\usepackage{wrapfig}

\usepackage[top=3cm,left=3cm,right=3cm,bottom=3cm]{geometry}
\usepackage{hyperref}
\usepackage{amssymb,latexsym,amsmath,epsfig,amsthm, tikz, verbatim, hyperref,color,amsthm,mathabx} 
\usepackage[english]{babel}
\usepackage{tikz-cd} 
\usepackage[section]{placeins}
\usepackage{diagbox}
\usepackage{pict2e}
\usepackage{hhline}
\usepackage{lineno}
\usepackage{caption}
\usepackage[all]{xy}

\makeatletter

\hypersetup{
	colorlinks=true,       % false: boxed links; true: colored links
	linkcolor=black,        % color of internal links
	citecolor=black,        % color of links to bibliography
	filecolor=black,     % color of file links
	urlcolor=blue
}
 
 \usepackage{framed}
\renewcommand\section{\@startsection {section}{1}{\z@}
{-30pt \@plus -1ex \@minus -.2ex}
{2.3ex \@plus.2ex}
{\normalfont\normalsize\bfseries}}
\renewcommand\subsection{\@startsection{subsection}{2}{\z@}
{-3.25ex\@plus -1ex \@minus -.2ex}
{1.5ex \@plus .2ex}
{\normalfont\normalsize\bfseries}}

\renewcommand{\@seccntformat}[1]{\csname the#1\endcsname. }

\makeatother
\newcommand{\CC}{\mathbb{C}}

\newcommand{\Conormal}{\operatorname{Con}}
\newcommand{\DL}{\operatorname{DL}}

\newtheorem{theorem}{Theorem}

\numberwithin{theorem}{section}
\newtheorem{proposition}[theorem]{Proposition}
\newtheorem{maintheorem}[theorem]{Main Theorem}

\newtheorem{lemma}[theorem]{Lemma}

\theoremstyle{definition}
\newtheorem{definition}[theorem]{Definition}
\newtheorem{remark}[theorem]{Remark}
\newtheorem{example}[theorem]{Example}


\begin{document}

\title{When does subtracting a rank-one approximation decrease tensor rank?}
\author{Emil Horobe\c{t} and Ettore Teixeira Turatti}

\date{} 

\maketitle
%\linenumbers
\begin{abstract}
Subtracting a critical rank-one approximation from a matrix always results in a matrix with a lower rank. This is not true for tensors in general. Motivated by this, we ask the question: what is the closure of the set of those tensors for which subtracting some of its critical rank-one approximation from it and repeating the process we will eventually get to zero? In this article, we show how to construct this variety of tensors and we show how this is connected to the bottleneck points of the variety of rank-one tensors (and in general to the singular locus of the hyperdeterminant), and how this variety can be equal to and in some cases be more than (weakly) orthogonally decomposable tensors. 
\end{abstract}

\section{Introduction}

Low-rank approximation of matrices is used for mathematical modeling and data compression, more precisely in principal component analysis, factor analysis, orthogonal regression, etc.
In order to get all critical rank-one approximations of a given matrix one can find all critical points of the distance function from the said matrix to the variety of rank-one matrices. By the Eckart-Young theorem, this is done by computing singular value decomposition. The number of such critical rank-one approximations (which is more generally equal to the \textit{Euclidean Distance Degree} of the variety of rank-one matrices \cite{DHOST16}) is always the minimum of the column and row dimensions of the matrix. 
Furthermore, by subtracting any such critical rank-one approximation from the matrix we get a drop in the rank, hence obtaining a suitable algorithm to construct any low-rank approximation of the matrix.

Low-rank approximations of tensors have even more application potential, but they are much more challenging both mathematically as well as computationally, tensor rank and many related problems are NP-hard, see \cite{Has90,HilLim13}. Despite this fact, many algorithms exist for finding rank-one approximations of a tensor. A way to do this, similarly to the matrix case, is by finding all critical points of the distance function from the said tensor to the variety of rank-one tensors (luckily this is an algebraically closed set). The generic number of such critical approximations was computed in \cite{FO} and shows that the degree of complexity of this problem for tensors is substantially higher than for matrices.

For higher-rank approximations, though, we have that tensors of bounded rank do not form a closed subset, so the best low-rank approximation of a tensor on the boundary does not exist (see \cite{DeSilvaLim}). From this, it results that subtracting a rank-one approximation from a tensor might even increase its rank (see \cite{SC10}). 

In this article, to resolve this obstacle for higher-rank approximations we turn our attention to the definition of \textit{border rank} of tensors (see definition for example in \cite{DraiKutt}) and we ask the question: what is the closure of the set of those tensors for which subtracting some of its critical rank-one approximation from it and repeating the process we will eventually get to zero?

We approach this problem by constructing the variety $\mathrm{DL}_2$ of tensors for which subtracting a critical rank-one approximation yields a rank-one tensor. Then we construct the variety $\mathrm{DL}_3$ of tensors for which subtracting a critical rank-one approximation yields an element of $\mathrm{DL}_2$, and so on. Our main finding is for the setting of symmetric tensors and it can be formulated as follows.

\begin{maintheorem}[Theorem~\ref{thm: stab} and Proposition~\ref{StepofStabilization}]

Let $X\subset \mathcal{S}^d\mathbb{C}^n$ be the cone over the Veronese variety and our inner product in $\mathcal{S}^d\mathbb{C}^n$ is the Bombieri-Weyl inner product. Then the chain $X=\DL_1\subset\dots\subset \DL_r\subset\dots$ stabilizes and the exact step of the stabilization is given by the formula in Proposition~\ref{StepofStabilization}.
\end{maintheorem}

Furthermore, for the general tensor case, we study in depth the variety $\mathrm{DL}_2$ of tensors for which subtracting one of its critical rank-one approximations we get a rank-one tensor. We will see that this variety is determined by the \textit{bottleneck points} of the variety of rank-one tensors and is related to the nodal singularities of the \textit{hyperdeterminant} (see Remark~\ref{NodalSing}). We also show the relation between  the $\DL_r$'s and weakly orthogonally decomposable tensors (see Definition~\ref{weaklyodeco}, Theorem~\ref{form_of_DL}, Proposition~\ref{PropforDL1} and Proposition~\ref{PropDLi}). Finally, in Section~\ref{Exa} we show examples of different behaviors of the limit variety $\DL_N$.

\section{Preliminaries}
Let us fix $n_1,n_2,\ldots,n_p\in \mathbb{N}$ and let $V=\mathbb{C}^{n_1}\otimes\ldots\otimes\mathbb{C}^{n_p}$. Now let us denote by $X\subseteq V$ the variety of rank at most one tensors of the form $x_1\otimes x_2\otimes\ldots\otimes x_p\in V$. This indeed is a variety and is always defined by the vanishing of all $2\times2$-subdeterminants of all flattenings of the tensor into matrices.
Such a flattening is obtained by partitioning the $p$ dimensions $n_1,\ldots, n_p$ into two
sets, for example $n_1,\ldots, n_q$ and $n_{q+1},\ldots, n_p$ and viewing the tensor as a $(n_1\cdot\ldots\cdot n_q)\times(n_{q+1}\cdot\ldots\cdot n_p)$-matrix.

Now we fix $\langle \cdot, \cdot\rangle_i$ inner products on the spaces $\mathbb{R}^{n_i}$, then it induces an inner product on $V_{\mathbb R}$ defined as follows.

\begin{definition}
The corresponding Bombieri-Weyl inner product on $V_{\mathbb R}=\mathbb{R}^{n_1}\otimes\ldots\otimes\mathbb{R}^{n_p}$ is defined for two rank-one tensors $x_1^1\otimes\ldots \otimes x_1^p, x_2^1\otimes\ldots \otimes x_2^p\in X$ by
\[\langle x_1^1\otimes\ldots \otimes x_1^p, x_2^1\otimes\ldots \otimes x_2^p\rangle_{BW}:=\prod_{i=1}^p\langle x_1^i,x_2^i\rangle_i\]
and it extends by linearity to $V_{\mathbb R}$.
Similarly for symmetric tensors, let $\langle\cdot,\cdot\rangle$ be an inner product in $\mathbb R^n$, the Bombieri-Weyl inner product corresponding to it in $\mathcal S^d\mathbb R^{n}$ is defined for two rank-one forms $u^d, v^d\in X$ by 
\[\langle u^d,v^d\rangle_{BW}:=\langle u,v\rangle ^d\]
and it extends by linearity to $\mathcal S^d\mathbb R^n$.
\end{definition}

Given an inner product $\langle \cdot,\cdot\rangle$ in $\mathbb R^n$, we will consider its extension as a bilinear form to the complex numbers. This means $\langle \cdot,\cdot\rangle$ is still a bilinear form, however, it is degenerated, i.e., there exists vectors $v\in \CC^n$ such that $\langle v,v\rangle=0$. Vectors with such property are called \emph{isotropic} vectors. For instance, if we consider the Euclidean metric in $\mathbb R^2$, then $v=(1,i)\in \CC^2$ is an example of an isotropic vector. 
Throughout the text, we will always consider the Bombieri-Weyl bilinear form in $\CC^{n_1}\otimes\dots\otimes\CC^{n_p}$ or $S^d\CC^n$ as the extension of a real inner product as in the previous paragraph.

Now take a tensor $T\in V$ and let $||\cdot ||$ be the norm corresponding to the chosen inner product on $V$, then we want to optimize the squared distance from $T$ to the variety $X$. So
\begin{equation}\label{MinProblem}
\begin{cases}
\text{minimize } d_T(x)=|| T-x||^2, \\
\text{subject to }x\in X.
\end{cases}
\end{equation}
The constrained critical points $x_1,\ldots,x_m$ of the function $d_T(x)$ are called the \textbf{critical rank-one approximations} of the tensor $T$. For a generic tensor $T$ the number $m$ of such constrained critical points is constant and is called the \textit{Euclidean Distance Degree} of the variety $X$ (for more details on this topic see \cite{DHOST16}) and it is given by a precise formula in \cite{FO}. 

Now we are interested to see what happens to the rank of the tensors $T-x_i$. It was shown in \cite{SC10}, that in the case of $n_1=n_2=n_3=2$, {so in $V=\mathbb{R}^2\otimes\mathbb{R}^2\otimes\mathbb{R}^2$} for a generic tensor $T$
(which has rank $2$) subtracting an $x_i$ results in a tensor $T-x_i$ that has
rank $3$. So the rank doesn't decrease but actually increases. 

Motivated by this phenomenon in continuation we are interested in the set of those tensors for which subtracting some of its critical rank-one approximation from it and repeating the process we will eventually get to zero.

\subsection{Joint ED correspondence and ED Duality}

Let us suppose that $X$ is minimally generated by $f_1,\ldots,f_s$ homogeneous polynomials. A classic approach to solving constrained optimization problems like \eqref{MinProblem} is to use Lagrange multipliers. Any constrained critical point of $d_T(x)$ is a solution to the following system, see \cite[Section 5.5.3]{ConvOpt},)
\begin{equation}\label{CritEqu}
\begin{cases}
\displaystyle\nabla d_T(x)+\sum_{i=1}^s \lambda_i \nabla f_i(x)=0\\
f_1(x)=f_2(x)=\ldots =f_s(x)=0.
\end{cases}
\end{equation}
Now, for regular points of $X$ we call the orthogonal complement (with respect to the chosen inner product) of the tangent space, $T_x X$, the normal space of $X$ at $x$, and we denote it by $N_xX$. For the above system, we will be interested only in solutions that are regular points (in this case only $0$ is a singular point) of $X$. We will denote the collection of regular points by $X_{reg}$ (and in general the singular points of a variety $X$ by $X_{sing}$).
Using this notation (by \cite{DHOST16}) the above system is equivalent to
\[
\begin{cases}
T-x \in N_x X,\\
x\in X_{reg}.
\end{cases}
\]
So a point $x\in X_{reg}$ is a solution to \eqref{CritEqu} if and only if there exists $y\in N_x X$, such that $x+y=T.$

We have seen so far that pairs of points $(x,y)$ such that $x\in X$ and $y\in N_xX$ play a crucial role in our analysis. The closure of the collection of all such pairs with $x\in X_{reg}$ is called the \textbf{conormal variety} and we denote it by $\Conormal(X)$. Formally we have
\begin{equation}\label{conormal}
\Conormal(X)=\overline{\{(x,y)\in V_x\times V_y\mid x\in X_{reg},\ y\in N_xX\}}.
\end{equation}
We use the notation $V_x\times V_y$ instead of just simply $V\times V$ to keep track that the first tuple of coordinates represents a point $x$ in $X\subseteq V$ and the second tuple of coordinates represents a point $y$ in $N_xX\subseteq V$.

There is a natural pair of projections $\pi_1:\Conormal(X)\to V_x$ to the first tuple of coordinates and $\pi_2:\Conormal(X)\to V_y$ to the second tuple of coordinates. The image of the first projection is the variety $X$ itself and the closure of the image of the second projection \[X^*:=\overline{\pi_2(\Conormal(X))}\] is called the \textbf{dual variety} of $X$ (see \cite[Section 5.4.2]{RS13}) or in other words the  \textbf{hyperdeterminant} of format $n_1\times\ldots\times n_p$ (see \cite{GKZ}).
This way we realize both $X$ and the dual $X^*$ in the same ambient space $V$.
\[\begin{tikzcd}
&\Conormal(X)\arrow{ld}[swap]{\pi_1}\arrow{rd}{\pi_2} \\
X\subseteq V& & X^*\subseteq V
\end{tikzcd}
\]

So we have that $x\in X_{reg}$ is a solution to \eqref{CritEqu} if and only if there exists a point $(x,y)$ in $\Conormal(X)$ such that $x+y=T$. To encapsulate this relationship we construct the so-called \textbf{joint ED correspondence} $\mathcal{E}_{X,X^*}$ (see \cite[Section $5$]{DHOST16}), to be the closure of all triples
\[
\{(x,y,T)\mid (x,y)\in \Conormal(X)\subseteq V, T\in V \text{ and } T=x+y\}.
\] So $\mathcal{E}_{X,X^*}$ is the closure of the graph of the Minkowski sum over $\Conormal(X)$. We have the following diagram of projections.

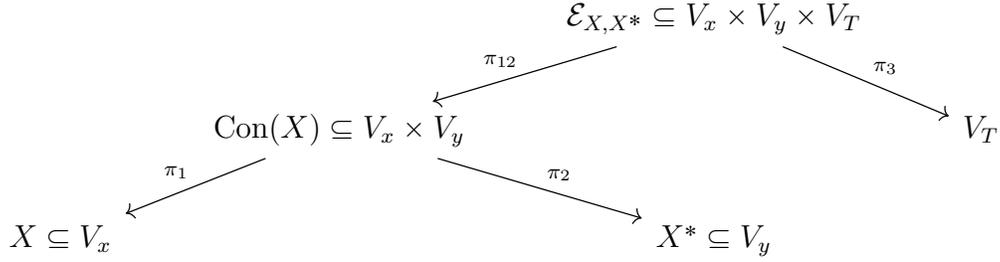
\[\begin{tikzcd}
& &\mathcal{E}_{X,X^*}\subseteq V_x \times V_y \times V_T \arrow{ld}[swap]{\pi_{12}}\arrow{rd}{\pi_3}& \\
&\Conormal(X)\subseteq V_x \times V_y\arrow{ld}[swap]{\pi_1}\arrow{rd}{\pi_2}& & V_T \\
X \subseteq V_x& & X^*\subseteq V_y& 
\end{tikzcd}
\]
Now we remind our reader of the duality property from \cite[Theorem $5.2$]{DHOST16} governed by the joint ED correspondence.
\begin{theorem}[ED Duality]\label{EDDuality}
Let $X\subseteq V$ be an irreducible affine cone, $X^*\subseteq V$ its dual variety, and $T\in V$ a general data point. The map $x\mapsto T-x$ gives a bijection from the critical points of $d_T$ on $X$ to the critical points of $d_T$ on $X^*$.
\end{theorem}
So this means that subtracting a critical rank-one approximation from $T$ results in a critical point of the distance function $d_T$ to the dual variety $X^*$. So in continuation, we will deal with tensors $T$ for which we will have additional requirements on the critical points of the distance function $d_T$ to the dual variety $X^*$.

\subsection{Special data locus on the dual variety $X^*$}

So we want to determine the set of all tensors $T\in V$, such that the optimization problem 

\begin{equation}\label{MinProblemDual}
\begin{cases}
\text{minimize } d_T(x)=||T-x||^2, \\
\text{subject to }x\in X^*.
\end{cases}
\end{equation}
on the dual variety $X^*$ has at least one critical point in a given subvariety of $X^*$ (namely $X$ to start with). 

By the ED duality \cite[Section $5$]{DHOST16} we have that the joint ED correspondence $\mathcal{E}_{X^*,X}$ for the distance optimization problem on the dual variety $X^*$ is the same as the joint ED correspondence $\mathcal{E}_{X,X^*}$ for the distance optimization problem on $X$, up to swapping the first to tuples $(x,y,T)\leftrightarrow(y,x,T)$ before taking the closures. This holds for the corresponding conormal varieties as well, so $\Conormal(X)$ is equal to $\Conormal(X^*)$ up to the swap $(x,y)\leftrightarrow(y,x)$, before taking the closures. So we have the following dual diagram of projections.

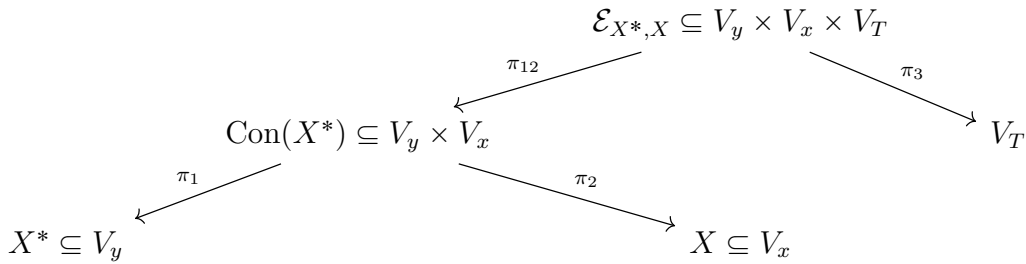
\[\begin{tikzcd}
& &\mathcal{E}_{X^*,X}\subseteq V_y \times V_x \times V_T \arrow{ld}[swap]{\pi_{12}}\arrow{rd}{\pi_3}& \\
&\Conormal(X^*)\subseteq V_y \times V_x\arrow{ld}[swap]{\pi_1}\arrow{rd}{\pi_2}& & V_T \\
X^* \subseteq V_y& & X\subseteq V_x& 
\end{tikzcd}
\]

For any subvariety $A\subseteq X^*$, following \cite{HR22}, we define the \textbf{data-locus of $A$} to be the closure of the projection $\pi_3$ into the space $V_T$ of tensors $T$, such that corresponding to such a tensor $T$ there is at least one pair of points $(y,x)\in\Conormal(X^*)$ in the conormal variety of $X^*$, such that $y\in A$. We will denote the data locus of $A$ by $\DL_A$. Formally we have that 
\begin{equation}\label{DefDLA}
\DL_A=\overline{\pi_{3}\left(\mathcal{E}_{X^*,X}\cap\left(A\times V_x\times V_T\right)\right)}.
\end{equation}
\begin{remark}
We choose not use the theorem in \cite[Theorem $5$]{HR22}, which describes the structure of a general $\DL_A$, because there we have the condition that $X^*_{reg}\cap A_{reg}\neq \emptyset$. For our discussion in general we can not have such an assumption on $A$.
\end{remark}

\section{The construction of the first layer of special tensor data locus $\DL_2$}
We continue by constructing the set of tensors $T$ for which subtracting a critical rank-one approximation we get a tensor that is at most rank-one. We denote the closure of this set by $\DL_2$.

So first we are interested in tensors $T$ with $x_1,\ldots,x_m$ critical rank-one approximations, such that $T-x_i\in X$ for some $x_i$. By the ED Duality (Theorem~\ref{EDDuality}) this is equivalent to searching for tensors $T$ with $T-x_1,\ldots,T-x_m$ critical points of the distance function $d_T$ from $T$ to the dual variety $X^* $, such that $T-x_i\in {X^\ast}$ for some $x_i$. So we want to determine the data locus of the subvariety $X\cap X^*=X\subseteq X^*$ of the distance optimization problem to $X^*$. So we have that \[\DL_2=\DL_{X\cap X^*}=\DL_X.\]
\begin{remark}\label{XsubvarX*}
Here we have that $X\subseteq X^*$, since any $x_1\otimes\ldots\otimes x_p$ rank-one tensor is perpendicular to the tangent space of $y=y_1\otimes\ldots\otimes y_p$, {namely $$T_{y}X=\langle \CC^{n_1}\otimes y_2\otimes\dots\otimes y_p,\dots,y_1\otimes \dots\otimes y_{p-1}\otimes\CC^{n_p}\rangle,$$}with some $x_i\perp y_i$ and $x_j\perp y_j$ (perpendicularity in at least two different indices) so $x_1\otimes\ldots\otimes x_p \in X^*$.
\end{remark}
Now we have the following proposition describing the structure of $\DL_2$
\begin{proposition}[Structure of $\DL_2$]\label{StructDL2}
Let $V=\mathbb{C}^{n_1}\otimes\ldots\otimes\mathbb{C}^{n_p}$ and let $X$ be the variety of rank-one tensors and $X^*$ its dual. Then \[\DL_2=\overline{\{x+y\mid x\in X_{reg}\text{ and }y\in N_xX\cap X.}\}\]
\end{proposition}
\begin{proof}
By definition~\ref{DefDLA} of the data locus we have that \[\DL_2=\overline{\pi_{3}\left(\mathcal{E}_{X^*,X}\cap\left((X\cap X^*)\times V_x\times V_T\right)\right)}.\]
So $\DL_2$ is the closure of \[\{y+x\mid (y,x,y+x)\in \mathcal{E}_{X^*,X},\ y\in X\cap X^*\}.\]

For any triple $(y,x,y+x)\in \mathcal{E}_{X^*,X}$, we have that $(y,x)\in \Conormal(X^*)$. Now there exists a sequence $(y_i,x_i)\to (y,x)$, with $y_i\in X^*_{reg}$ and $x_i\in N_{y_i}X^*$. We have that generically $x_i$ is a regular point of $X$ (here actually $X_{sing}=\{0\}$) so by applying the swap from the ED duality (Theorem~\ref{EDDuality}) we get that $y_i\in N_{x_i} X$, that is $(x_i,y_i)\in\Conormal(X)$. But now since $\Conormal(X)$ is closed, by taking limits we get that $(x,y)\in\Conormal(X)$, hence $y\in N_xX$. 
So we have that 
\[\DL_2\subseteq\overline{\{x+y\mid x\in X_{reg}\text{ and }y\in N_xX\cap X\}}.\]
For the other inclusion, take any pair $(x,y)$, such that $x\in X_{reg}\text{ and }y\in N_xX\cap X$. Then it is clear that $x$ is a critical point of $x+y$, because $y\in N_{x}X$ and subtracting $x$ from $x+y$ we get $y$ which is an element of $X$. Hence $x+y\in \DL_2$.
\end{proof}
\begin{remark}\label{NodalSing}
 We have that if for a pair $x,y\in X$, with $y\in N_xX$ in addition we require that $x$ and $y$ are projectively distinct points and also that $x\in N_y X$, then the given pair $(x,y)\in \mathrm{BN}(X)$, where by $\mathrm{BN(X)}$ we denote the \textbf{projective bottleneck pairs} of $X$ (see \cite[Lemma $2.4$]{DiRocco}). So we have that
\[\{x+y\mid (x,y)\in \mathrm{BN}(X)\}\subseteq \DL_2.\]
Moreover, we have that the projective bottleneck points are corresponding to the nodal singularities of $X^*$ (see \cite[Theorem $8.14$]{Tev} or \cite[Theorem $0.3$]{WZ}). We will see in the upcoming examples, that indeed the structure of the singular locus of $X^*$ plays an important role in our discussion.
\end{remark}
We will continue with presenting a set of properties of $\DL_2$. We start with the dimension of the data locus. First, we want to understand how the dimension of $\DL_2$ relates to the dimension of $X$. For this, we have to distinguish two separate cases. First if $\dim(X\cap X^*_{sing})<\dim (X)$ or if $\dim(X\cap X^*_{sing})=\dim (X)$. In the second case, we can not say anything about the dimension.

For the first case though we get the following general description for the dimension of data loci.
\begin{proposition}[Dimension of $\DL_A$ non-singular case]\label{DimofDL}
Let $X\subseteq \mathbb{C}^n$ be a cone and $X^*$ its dual. Furthermore let $A\subset X$ be a proper subvariety s.t. $\dim(A\cap X_{sing})<\dim A$. Then we get that \[\dim \DL_A=\dim A+\mathrm{codim} X.\] 
\end{proposition}
\begin{proof}

In this case, we can use the structure theorem of $\DL_A$ according to \cite[Theorem $5$]{HR22} and we get that $\DL_A$ is the closure of the image under the Minkowski sum of
	\[\Conormal(A)\cap\Conormal(X)=\{(x,y)\mid x\in A\setminus X_{sing},\ y\in N_x X\}.\]
Just like in the proof of \cite[Theorem $4.1$]{DHOST16} this means that $\pi_1:\Conormal(A)\cap\Conormal(X)\to A$ is an affine vector bundle of rank equal to the codimension of $X$, since the fiber over an $x\in A\setminus X_{sing}$ is equal to $\{x\}\times N_xX$, where the second factor is an affine space of dimension equal to the codimension of $X$, varying smoothly with $x$. Because we have that $A\cap X_{sing}$ is of codimension higher than zero in $A$, we get that the dimension of $\Conormal(A)\cap\Conormal(X)$ is equal to $\dim A+\mathrm{codim} X$. Now the Minkowski sum map $\Sigma:\Conormal(A)\cap\Conormal(X)\to \DL_A$, with $(x,y)\mapsto x+y$, cannot have positive dimensional fibers over a general point $T\in\DL_A$ because, for $(x,y)\in \Sigma^{-1}(T)$, the corresponding $x$ is a critical point of the distance function from $T$ to $X$, hence the number of such $x$'s is finite and equal to the ED degree of $X$. 
So in this case the dimension of $\DL_A$ equals the dimension of $\Conormal(A)\cap\Conormal(X)$, hence the claim.
\end{proof}

\begin{remark}
In particular for $\DL_2$, by applying the above proposition to $X^*$ the dual variety to $X \subseteq \mathbb{C}^{n_1}\otimes\ldots\otimes\mathbb{C}^{n_p}$ the cone over rank-one tensors  and setting $A=X$ 
 ( by Remark~\ref{XsubvarX*}) we get that if we have that $\dim(X\cap X^*_{sing})<\dim (X)$, then \[\dim \DL_2=\dim (X)+\mathrm{codim}\ X^*.\]
\end{remark}

Now we have the following set of properties of $\DL_2$, which are about the relation of $\DL_2$ to the variety of rank-one tensors, to border rank-two tensors, and to (weakly) orthogonally decomposable tensors. In order to continue, based on the classical definition of orthogonally decomposable tensors (see \cite{AnnaElina}) and generalizing the definition of weakly orthogonally decomposable tensors in the symmetric setting (see \cite{BiDraSe}), we give the following definition.
\begin{definition}\label{weaklyodeco}
A tensor $T\in V=\mathbb{C}^{n_1}\otimes\ldots\otimes\mathbb{C}^{n_p}$ is \textbf{weakly orthogonally decomposable} if $T$ can be written as
\[T=\sum_{i=1}^m x_1^i\otimes\ldots\otimes x_p^i,\] with $x_i^j$ are nonzero and $x_j^k\perp x_j^l$, with $k\neq l$, for all $1\leq j\leq p$. If moreover, we get that $x_j^l\notperp x_j^l$, so no $x_j^l$ is isotropic, then $T$ is called strongly orthogonally decomposable. Moreover, if a real strongly orthogonally decomposable tensor has such a  decomposition with real terms, then it is classically orthogonally decomposable.
\end{definition}

\begin{proposition}\label{PropforDL1}
Let $V=\mathbb{C}^{n_1}\otimes\ldots\otimes\mathbb{C}^{n_p}$ and let $X$ be the cone over the variety of rank-one tensors and $X^*$ its dual. Then we have that 
\begin{itemize}
	\item[a. ] $X\subseteq \DL_2$ and as a consequence (for the sake of consistency in notation) we will denote $\DL_1:=X$;
	\item[b. ] Any $T\in \DL_2$ has border rank less than or equal to $2$;
	\item[c. ] $\DL_2$ contains all rank at most $2$ weakly orthogonally decomposable tensors.
\end{itemize}
\end{proposition}
\begin{proof}
\begin{itemize}
\item[a. ] First we observe that $\{0\}\in X$, so $\DL_{\{0\}}\subseteq \DL_2$, now by applying \cite[Theorem $1$]{H17} and following biduality we have that $X=(X^*)^*\subseteq \DL_{\{0\}}\subseteq \DL_2$. Hence the inclusion. 
	(We remark that the $X\subseteq  \DL_2$ inclusion also follows from  \cite[Corollary $3.2$]{HR22}.)
	  
	\item[b.] For any $T\in \DL_2$, there exists a sequence of $T_i\to T$, such that for each $T_i$ there exists an $x_i\in X$ with ${T_i}-x_i\in X$ and this way $T_i=x_i+(T-x_i)$, so $T_i$ is of rank at most $2$. But by definition, this exactly means that $T$ is of border rank at most $2$ (is the limit of rank at most $2$ tensors). 
	
	\item[c.] By construction, all rank at most two weakly orthogonally decomposable tensors $T$ are of the form $x_1^1\otimes\ldots\otimes x_p^1+x_1^2\otimes\ldots\otimes x_p^2$, with $x_i^1\perp x_i^2$. It is classical, that the tangent space to $X$ at the point $x_1^1\otimes\ldots\otimes x_p^1$ is equal to $\sum_{i=1}^p x_1^1\otimes\ldots\otimes \CC^{n_i}\otimes \ldots\otimes x_p^1$. Now by \cite[section 2.2]{DOT18} we have that for any $i$ and for any $y\in\CC^{n_i}$ we get 
 \[\langle x_1^1\otimes\ldots\otimes y\otimes \ldots\otimes x_p^1, x_1^2\otimes\ldots\otimes x_p^2\rangle_{BW}=\prod_{j\neq i}\langle x_j^1,x_j^2\rangle_j \langle {y,x_i^2}\rangle_i=0.\] This means that $x_1^2\otimes\ldots\otimes x_p^2\perp T_{x_1^1\otimes\ldots\otimes x_p^1} X$, so {$x_1^1\otimes\ldots\otimes x_p^1$} is a critical point of $x_1^1\otimes\ldots\otimes x_p^1+x_1^2\otimes\ldots\otimes x_p^2$.
\end{itemize}
\end{proof}
\begin{proposition}
    We have that $\DL_2$ always has ${p\choose 2}$ irreducible components and their intersection contains all rank at most two weakly orthogonally decomposable tensors.
\end{proposition}
\begin{proof}
    Let $x_1=x_1^1\otimes\dots\otimes x_1^p, x_2=x_2^1\otimes\dots\otimes x_2^p\in V=\CC^{n_1}\otimes\dots\otimes\CC^{n_p}$ be rank-one tensors. Suppose $x_1+x_2\in \DL_2$ is a generic element, then $x_1\perp T_{x_2}X=\sum_{i=1}^p x_2^1\otimes \ldots\otimes \CC^{n_i}\otimes\ldots\otimes x_2^p$. For all $1\leq j\leq p$ this leads to the conditions 
    \[\prod_{i\neq j}\langle x_1^i,x_2^i\rangle_i=0.\]
All of these conditions are satisfied simultaneously when there are at least two indices $j_1, j_2$ such that $\langle x_1^{j_1}, x_2^{j_1}\rangle_{j_1}=\langle x_1^{j_2}, x_2^{j_2}\rangle_{j_2}=0$. $\DL_2$ is given by the union over all pairs of indices $j_1, j_2$ of the spaces defined by the previous conditions, and there are  $p\choose 2$ pairs.
\end{proof}
\section{The sequence of higher border rank special tensor data loci}
Now to continue our pursuit we construct the second layer of data locus, the $\DL_3$, the variety of those tensors for which subtracting a critical rank-one approximation we get a tensor on $\DL_2$. By the ED Duality (Theorem~\ref{EDDuality}) this is equivalent to searching for tensors $T$ with $T-x_1,\ldots,T-x_m$ critical points of the distance function $d_T$ from $T$ to the dual variety $X^* $, such that $T-x_i\in \DL_2$ for some $x_i$. So we want to determine the special data locus of the subvariety $\DL_2\cap X^*\subseteq X^*$ of the distance optimization problem to $X^*$. This means that we have that $\DL_3=\DL_{X^*\cap\DL_2}$. 
We continue with the construction of $\DL_3,\DL_4,\ldots$ so we define in general
\[\DL_r=\DL_{X^*\cap\DL_{r-1}}.\]  
We have the following proposition describing the general structure of the $\DL_r$. The proof of this statement goes in a similar fashion to the proof of Proposition~\ref{StructDL2}.
\begin{proposition}[Structure of $\DL_r$]\label{StructDLi}
Let $V=\mathbb{C}^{n_1}\otimes\ldots\otimes\mathbb{C}^{n_p}$ and let $X$ be the cone over the variety of rank-one tensors and $X^*$ its dual. Then \begin{equation*}\label{eq:DLr}
\DL_r=\overline{\{x+y\mid\text{such that }x\in X_{reg}\text{ and }y\in N_xX\cap \DL_{r-1}.}\}\end{equation*}
\end{proposition}
We have the following properties of $\DL_r$.
\begin{proposition}\label{PropDLi}
Let $V=\mathbb{C}^{n_1}\otimes\ldots\otimes\mathbb{C}^{n_p}$ and let $X$ be the cone over the variety of rank-one tensors and $X^*$ its dual. Then we have that 
\begin{itemize}
	\item[a.] $\DL_r\subseteq \DL_{r+1}$;
	\item[b.] Any $T\in \DL_r$ has border rank at most $r$;
	\item[c.] $\DL_r$ contains all rank at most $r$ weakly orthogonally decomposable tensors.
\end{itemize}
\end{proposition}

\begin{proof}
\item[a.]{First, we have that $\DL_1\subset \DL_2$. Indeed, let $x\in X_{reg}$, then $x=x+0\in \DL_2$, as $x\in X_{reg}$ and $0\in N_xX\cap X$. So $X_{reg}\subset \DL_2$. Then, by taking closures, $X=\DL_1\subset \DL_2$.\\
We now show that $\DL_2\subset \DL_3$. Let $x+y\in \DL_2$ not on the boundary, then $x\in X_{reg}$ and $y\in N_xX\cap \DL_1=N_xX\cap X$. Since $y\in \DL_1$, by the previous step we have that $y\in \DL_2$, so $y\in N_xX\cap \DL_1$ implies $y\in N_xX\cap \DL_2$. Therefore, $x\in X_{reg},\ y\in N_xX\cap\DL2$, so $x+y\in \DL_3$. Again by taking closures, $\DL_2\subset \DL_3$.\\
To conclude, assume by induction that $\DL_{r-1}\subset \DL_{r}$, and let $x+y\in \DL_r$ not in the boundary, so $x\in X_{reg}$ and $y\in\DL_{r-1}\cap N_xX\subset \DL_r\cap N_xX$, thus $x\in X_{reg}$ and $y\in \DL_r\cap N_xX$, then $x+y\in \DL_{r+1}$. Taking closures gives $\DL_r\subset \DL_{r+1}$.}
We remark that the $\DL_r\cap X^*\subseteq  \DL_{r+1}$ inclusion also follows from  \cite[Corollary $3.2$]{HR22}. 
		%\item[a.] For any tensor $T$ which is not on the boundary of $\DL_r$, subtracting a critical rank-one approximation from itself by definition we get a tensor on $\DL_{r-1}$, which by induction is a subvariety of $\DL_r$, hence after subtraction we get a tensor on $\DL_r$. So the original  $T$ is in $\DL_{r+1}$. But now also the algebraic closure of such tensors must be in $\DL_{r+1}$ because this latter one is algebraically closed. Hence the inclusion. 

	\item[b.] {Let $T=x+y\in \DL_2$ not in the boundary, then $x,y\in X$, so $T$ has rank at most two. Taking closures implies that $\DL_2$ is contained in the second secant variety of $X$.} For any tensor $T$ which is not on the boundary of $\DL_r$, there exists an $x\in X$ critical approximation such that $T-x\in\DL_{r-1}$. By induction $T-x$ has border rank at most $r-1$, so there exists a sequence $T_n\to T-x$, with each $T_n$ having rank at most $r-1$. This way the sequence $T_n+x$ converges to $T$ and all $T_n+x$ have rank at most $r$. So $T$ is of border rank at most $r$. But having border rank at most $r$ is an algebraically closed condition, so all tensors $T\in \DL_r$ have border rank at most $r$;
	
	\item[c.] By construction all rank at most $r$ weakly orthogonally decomposable tensors $T$ are of the form $\displaystyle{\sum_{k=1}^{r}x_1^k\otimes\ldots\otimes x_p^k}$, with $x_j^k\perp x_j^l$, for all $k\neq l$ and the proof goes as in Proposition~\ref{PropforDL1}, part $c.$
\end{proof}

%Now we switch our attention to the sequence $X=\DL_1\subseteq \DL_2\subseteq\ldots\subseteq V$. We are interested in those instances when the above sequence stabilizes.

%\begin{remark} We start with the following remark.
%\begin{itemize}
	%\item[a.] The minimum number of steps to reach the limit $\DL_N$ is equal to the minimum of the dimensions $n_1,\ldots,n_p$. This is because first of all by their definition orthogonally decomposable tensors are a subvariety of $\DL_N$. Then by Proposition~\ref{PropDLi} point $(c)$ we have that at step $i$ in $\DL_N$ we get included all the rank $i$ orthogonally decomposable tensors. And finally, the maximal rank of an orthogonally decomposable tensor is the minimum of the dimensions $n_1,\ldots,n_p$.
	%\item[b.] By Proposition~\ref{PropDLi} point $(b)$ we have that any tensor $T\in \DL_N$ has border rank at most $N$.
	%\item[c.] By Proposition~\ref{DimofDL} we can see that we will not get stabilization of the sequence $\DL_i$ whilst we have that $\dim((X\cap X^*)\cap X^*_{sing})<\dim (X\cap X^*)$, because in this case $\dim \DL_{i}<\dim \DL_{i+1}$. So $\dim((X\cap X^*)\cap X^*_{sing})=\dim (X\cap X^*)$ is a necessary condition for stabilization of the sequence. But is not a sufficient condition for the stabilization, you can see this for instance on Example~\ref{RegualEx}, where we already started the sequence with the subvariety $X\cap X^*=X\subseteq X^*_{sing}$ and there is no stabilization yet at that step.
	
%\end{itemize}	
%\end{remark}
\section{Symmetric tensors}

Now let us switch our attention to the space of symmetric tensors. Let us denote the closure of the set of all linear forms \[\{v_1 x_1+\ldots+v_n x_n\mid (v_1,\ldots,v_n)\in \mathbb{C}^n\}\]  by $\mathcal{S}^1\mathbb{C}^n$ and the variety of all homogeneous degree $d$ polynomials in $n$ variables by $\mathcal{S}^d\mathbb{C}^n$. Let $X$ be the closure of the $d$-th powers of linear forms in $\mathcal{S}^d\mathbb{C}^n$, namely the cone over the Veronese variety.

\begin{proposition}\label{prop: isotropic}
    Suppose $l_1,\dots,l_r\subset \mathcal S^1\CC^n$ are all isotropic and orthogonal to each other. Then for any $l\in \langle l_1,\dots,l_r\rangle_{\CC}$ the resulting $l^d$ is a critical rank-one approximation of $f=\displaystyle{\sum _{i=1}^rl_i^d}$.
\end{proposition}
\begin{proof}
    Let $l=\displaystyle{\sum_{i=1}^{r}}\alpha_il_i$, with $\alpha_i\in\CC$, then $l^d$ is a critical rank-one approximation of $f$ if and only if $f-l^d\perp T_{l^d}X$. We recall that the tangent space at $l^d$ to $X$ is $\{l^{d-1}v \mid v\in\mathcal{S}^1\CC^n\}.$ So we want to check that
    \[\langle l^{d-1}v,f-l^d\rangle_{BW}=\langle l^{d-1}v,f\rangle_{BW}-\langle l^{d-1}v,l^d\rangle_{BW}=0,\] for every $v\in \mathcal{S}^1\CC^n$. But now notice that since $l$ is isotropic, then by the property in \cite[Section $2$, Page $4$]{DOT18} we get that \[\langle l^{d-1}v,l^d\rangle_{BW}=\langle l,l\rangle^{d-1}\langle v,l\rangle=0.\]
    So it is enough to show that $f\perp T_{l^d}X$. Indeed 
%$$\langle l^{d-1}v,f\rangle=\sum_{i=1}^r\langle l^{d-1}v,l_i^d\rangle=\sum_{i=1}^r\sum_j \gamma_j\langle l_1^{\beta_{j,1}}\dots l_{r'}^{\beta_{j,r'}},l_i^d\rangle\langle v,l_i\rangle=0,
%$$
%with $\sum_i\beta_{j,i}=d-1,\ \gamma_j=\gamma_j(\alpha_1,\dots,\alpha_r)\in\CC$.
\[\langle l^{d-1}v,f\rangle_{BW}=\sum_{i=1}^r\langle l^{d-1}v,l_i^d\rangle_{BW}=\sum_{i=1}^r\langle l,l_i\rangle^{d-1}\langle v,l_i\rangle=0,\]
because $l=\displaystyle{\sum_{i=1}^{r}}\alpha_il_i$ and $l_i$ are all isotropic and pairwise orthogonal, so we get that $\langle l,l_i\rangle=0$, for all $i=1,\ldots,r.$
\end{proof}

\begin{theorem}\label{form_of_DL}
    For any $r$ we get that $\DL_r$ is the variety of weakly orthogonally decomposable tensors of rank at most $r$, namely 
    \[\DL_r=\overline{\{l_1^d+\dots+\l_r^d\mid l_i\in\mathcal{S}^1\CC^n,\ l_i\perp l_j,\ i\neq j\}}.\] 
\end{theorem}
\begin{proof}
First, for the $"\supseteq"$ inclusion follows from Proposition~\ref{PropDLi} part $c.$ We continue by proving the $"\subseteq"$ inclusion. First, by Proposition~\ref{StructDL2} we get that \[\DL_2=\overline{\{l_1^d+l_2^d\mid l_1^d\perp T_{l_2^d}X\}},\] from this follows that $\langle l_1^d,l_2^d\rangle_{BW}=0$, so $l_1\perp l_2$ and so $\DL_2$ has the required structure. Now we proceed by induction. We assume that \[\DL_r=\overline{\{l_1^d+\dots+\l_r^d\mid l_i\in\mathcal{S}^1\CC^n,\ l_i\perp l_j,\ i\neq j\}}.\] 
Now by Proposition~\ref{StructDLi} we get that $\DL_{r+1}$ is the closure of $\{f+l^d | f\in \DL_r, f\perp T_{l^d}X \}$. This is equivalent to saying that generically $f$ has the form $f=\sum_{i=1}^r l_i^d$ with $l_i\perp l_j$, and $f\perp T_{l^d}X=\{l^{d-1}v|v\in\mathcal S^1\CC^n \}$. Thus \begin{equation}\label{eq: DL_r prod}
       \langle f,l^{d-1}v\rangle_{BW}= \sum_{i=1}^r\langle l_i^d,l^{d-1}v\rangle=\sum_{i=1}^r\langle l_i,l\rangle^{d-1}\langle l_i,v\rangle=0
    \end{equation}
    for all $v\in\mathcal S^1\CC$, where equality holds by \cite[section 2.2]{DOT18}.

If $l_i$ is not isotropic, by considering $v=l_i$ {in \eqref{eq: DL_r prod} and utilizing $l_i\perp l_j$ for every $j\neq i$, \eqref{eq: DL_r prod} simplifies to $\langle l, l_i\rangle^{d-1}=0$, so $l_i\perp l$}. We are left to deal with the case that some $l_i$'s are isotropic, suppose that $l_1,\dots,l_{r'}$ is a maximally linearly independent set among the isotropic linear forms in the sum defining $f$. Notice that it is enough to prove that $l\perp l_1,\ldots,l_{r'}$ because then $l$ is perpendicular to everything in their span (hence to the other isotropic linear forms in the sum defining $f$ that are linearly dependent of $l_1,\ldots,l_{r'}$ ).

Now there exists a $u\in\mathcal{S}^1\CC^n$, such that $u\perp l_2,\dots,l_{r'}$ but $u\notperp l_1$. Because otherwise if for all $u\in\mathcal{S}^1\CC^n$, such that $u\perp l_2,\dots,l_{r'}$ it would follow that $u\perp l_1$ as well, then this would mean that $\displaystyle{\langle l_2,\ldots,l_{r'}\rangle_{\CC}^{\perp}}\subseteq \langle l_1\rangle_{\CC}^{\perp}$, which in particular implies that $\langle l_1\rangle_{\CC}\subseteq \langle l_2,\ldots,l_{r'}\rangle_{\CC}$ contradicting the linear independence of the $l_i's$.

%Since they are linearly independent and for any two vectors $u,v\in \CC^n$, if $u^\perp=v^\perp$, then $u\in\langle v\rangle_\CC$, it follows that there exists $u\perp l_2,\dots,l_r$ but $u\notperp l_1$ (otherwise they wouldn't be linearly independent and orthogonal to each other). 

Now by selecting $v=u$ in \eqref{eq: DL_r prod}, we obtain $\langle l_1,l\rangle^{d-1}=0$, hence $l_1\perp l$. The same argument holds for $i=2,\dots,r'$. {We just proved $l\perp l_i$ for every $i=1,\dots,r$}, so it follows that 

\[\DL_{r+1}=\overline{\{ l_1^d+\dots+l_{r+1}^d\mid l_i\perp l_j, i\neq j \}}.\]

%Moreover, notice that a generic $f=\sum_{i=1}^rl_i^d\in \DL_r^\circ$ has rank $r$. Indeed, if we have that $\dim\langle l_1,\dots,l_r\rangle_\CC=r$ (in particular, this means that non of the $l_i$'s is isotropic) and since $l_i\perp l_j$, it is easy to see that the rank of $f$ is equal to $r$ (for example by applying a linear change of coordinates $\CC[x_1,\ldots,x_n]\to \CC[y_1,\ldots,y_n]$ of the form $y_i=l_i$ and applying \cite[Theorem 1.1]{CCC15}). 

%If $\dim\langle l_1,\dots,l_r\rangle_\CC<r$, then there is at least one $l_i$ that is isotropic. Of course if $l_1$ is isotropic, then a redundant decomposition $f=\frac{l_1^d}{2}+\frac{l_1^d}{2}+\sum_{i=3}^r l_i^d$ satisfies the hypothesis, but in such a case $f\in \DL_{r-1}$, hence $f$ is not generic in $\DL_r$ (\Emil{I think this is not necessarily true, here we assume that the dimension of $\DL_{r-1}$ is smaller than of $\DL_r$, which is not necessarily true}), unless $\DL_r=\DL_{r-1}$ in which case we are done already. So we assume the decomposition is non-redundant, $\DL_{r-1}\neq\DL_r$ and $l_1,\dots,l_{t}$ are isotropic. Suppose $l_1,\dots,l_{t}\in\langle l_1,\dots,l_s\rangle_\CC$, $s<t$. Then we may consider $l_1,\dots,l_s,l_{t+1},\dots,l_r$ essential variables to $f$, and $f=h(l_1,\dots,l_s)+\sum_{i=t+1}^rl_i^d$. A minimal decomposition of $h$ will lead to a minimal decomposition of $f$ by \cite[Theorem 1.1]{CCC15}, and $\mathrm{rk}(h)\leq t$. If $\mathrm{rk}(h)<t$, then $f\in \DL_{r-1}^\circ$, violating the genericity. Thus $\mathrm{rk}(f)$=r.
\end{proof}

\begin{theorem}\label{thm: stab}
    Let $X\subset \mathcal{S}^d\mathbb{C}^n$ be the cone over the Veronese variety and our inner product in $\mathcal{S}^d\mathbb{C}^n$ is the Bombieri-Weyl inner product. Then the chain $X=\DL_1\subset\dots\subset \DL_r\subset\dots$ stabilizes. %In particular, $\DL_n$ contains all real orthogonally decomposable tensors.
\end{theorem}
\begin{proof}

Notice that by starting with a generic element $f\in \DL_r$ and adding $l^d$ such that it is a critical rank-one approximation of $f+l^d$, we obtain a generic element of $\DL_{r+1}$, thus we only need to show this process stabilizes for a generic $f\in \DL_r$.

By Theorem~\ref{form_of_DL} we may assume that $f$ has the form $f=\sum_{i=1}^{r'} l_i^d+\sum_{j=1}^{r''}y_j^d$, $r'+r''=r$, where each summand is orthogonal to each other, where $l_i$'s are isotropic and $y_j$'s are not isotropic. {Now if $l$ is not isotropic, then $l$ is not in the span of the $y_j$'s, otherwise, $l\notperp y_j$ for some $j$ and $f+l^d\notin \DL_{r+1}$}. Hence the dimension of the span of the non-isotropic terms increases. This can not happen infinitely many times, so this case will eventually stop happening. So we can now assume that $l$ is isotropic. Now again we have two cases. Either $l\notin \langle l_1,\ldots,l_{r'}\rangle_{\CC}$ or $l\in\langle l_1,\ldots,l_{r'}\rangle_{\CC}$. In the first case, the dimension of the span of the $l_i$'s will increase, but this can not happen infinitely many times, so this case will eventually stop happening. So we can now assume that we are in the second case when $l\in\langle l_1,\ldots,l_{r'}\rangle_{\CC}$ and indeed $l$ can be freely chosen from this span by Proposition \ref{prop: isotropic}. We now have to show that also this case will eventually stop happening.

Let us suppose that $l_1,\ldots,l_{s}$ is a maximal linearly independent subset among the $l_i$'s. Then observe that in this case $f+l^d$ can be realized as a polynomial in the essential variables $l_1,\dots,l_{s},y_{1}\dots,y_{r''}$. So we get that
%$\dim\langle l_1,\dots,l_{r'}\rangle=r'$ and $\langle l_1,\dots,l_{r'},y_{r'+1},\dots,y_r\rangle^\perp=\langle l_1,\dots,l_{r'}\rangle$ (\Emil{Once again, why is this true? What if there are no $l_i's$?})(we could arrive at the same setting with finitely many steps, maybe with some linear dependent isotropic linear forms, this is exactly the behavior we deal next). First, notice $f$ can be realized as a polynomial in the essential variables $l_1,\dots,l_{r'},y_{r'+1}\dots,y_r$. Now let $f+l^d\in \DL_{r+1}^\circ$, by repeating the computations to determine $\DL_r$ we have $l^d\perp l_i^d,y_j^d$ for all $1\leq i\leq r'<j\leq r$, from our assumptions 

%Notice that $\mathrm{rk}(f+l^d)\in\{r-1,r,r+1\}$, if it is equal to either $r-1$, $r$, then $f+l^d$ has minimal decomposition $$
\[f+l^d=\sum_{i=1}^{r'+1} l_i^d+\sum_{j=1}^{r''} y_j^d,\text{ with }l_i\in \langle l_1,\ldots,l_{s}\rangle_{\CC},\ l=l_{r'+1}.
\] 
Moreover, for any $l_i,l_j\in \langle l_1,\ldots,l_{s}\rangle_{\CC}$ we get that $l_i\perp l_j$ and also $l_i\perp y_k$, for any $1\leq k\leq r''$.
Now by \cite[Theorem $1.1$ or Corollary $3.2$]{CCC15} the rank of $f+l^d$ is equal to the rank of $\sum_{i=1}^{r'+1} l_i^d$ plus $r''$. Now we again have two cases. 

First, the rank of $f+l^d$ is less than or equal to the rank of $f$. In this case, $f+l^d$ has a minimal decomposition of the form
\[f+l^d= \sum_{i=1}^{t} g_i^d+\sum_{j=1}^{r''} y_j^d,\text{ with }g_i\in \langle l_1,\ldots,l_{s}\rangle_{\CC},\] where $t\leq r'$. Moreover, since $g_i\in \langle l_1,\ldots,l_{s}\rangle_{\CC}$, then for any $g_i,g_j$ we also get that $g_i\perp g_j$ and also $g_i\perp y_k$, for any $1\leq k\leq r''$. So this means that $f+l^d$ has such a sum decomposition, that actually $f+l^d\in \DL_{t+r''}\subseteq \DL_{r}$, since $t+r''\leq r'+r''=r$. So in this case we do not get an new element in $\DL_{r+1}$. 

Second, the rank of $f+l^d$ is bigger than the rank of $f$ and $f+l^d\in\DL_{r+1}\setminus \DL_r$. It is clear that this will be the typical behavior while $r'$ is less than the generic rank $g$ in $\mathrm{S}^d\CC[l_1,\dots,l_{s}]$. So this case will happen until $r'=g$. Now we can assume that $f\in \DL_{g+r''}$. Notice that if we did not yet achieve stabilization till this step, then there must exist an element of $\DL_{g+r''}$ with rank equal to $g+r''$. Otherwise, if all elements of $\DL_{g+r''}$ would have rank less than $g+r''$, then we may proceed as in the previous paragraph and it would turn out that all elements of $\DL_{g+r''}$ are already in some previous data loci. So we may further assume that $\mathrm{rk}(f)=g+r''$. Then if $f+l^d$ has rank $g+r''+1$, for $l\in \langle l_1,\dots,l_{s}\rangle_\CC$, then $h=\sum_{i=1}^g l_i^d+l^d$ has rank $g+1$. Since $g$ is the generic rank in $S^d\CC[l_1,\dots,l_s]$, there exists a sequence $f_k=\sum_{i=1}^g L_{i,k}^d$, $L_{L_{i,k}}\in\langle l_1,\dots,l_{s}\rangle_\CC$ such that $f_k$ converges to $h$. This means $f_k+\sum_{i=1}^{r''} y_i^d\in \DL_{g+r''}$ converges to $f+l^d$, thus $f+l^d\in\DL_{g+r''}$, and so the procedure stabilizes.

\end{proof}
\begin{lemma}\label{prop: maxspaniso}
    The largest linear subspace spanned by orthogonal isotropic vectors in $\CC^n$ has dimension $\left\lfloor \frac{n}{2}\right\rfloor$.
\end{lemma}
\begin{proof}
    First, we notice that we indeed can have $m=\left\lfloor \frac{n}{2}\right\rfloor$ linear independent and orthogonal isotropic vectors, for instance, one can consider  the vectors $w_1=(1,i,0\dots,0), w_2=(0,0,1,i,0,\dots,0),\dots,w_m=(0,\dots,0,1,i)$ or $w_m=(0,\dots,0,1,i,0)$ depending on the parity of $n$. 
    
    Let $v_1,\dots,v_m$ be linearly independent and pairwise orthogonal isotropic vectors. If $n$ is even, then $\langle v_1,\dots,v_m\rangle_\CC=\langle v_1,\dots,v_m\rangle_\CC^\perp$, since it is the intersection of $m$ independent hyperplanes. If there was $v\perp \langle v_1,\dots,v_m\rangle_\CC$ and $v_1,\dots,v_m,v$ linearly independent, then $\dim \langle v_1,\dots,v_m,v\rangle_\CC^\perp<m$, a contradiction since $\langle v_1,\dots,v_m\rangle_\CC$ is a subset of it.

    If $n$ is even, then $\langle v_1,\dots,v_m\rangle_\CC^\perp=\langle v_1,\dots,v_m,v\rangle_\CC$, with $v_1,\dots,v_m,v$ linearly independent. If $v$ was isotropic, then $\langle v_1,\dots,v_m,v\rangle_\CC^\perp=\langle v_1,\dots,v_m,v\rangle_\CC$ has dimension $m+1$, however since $v_1,\dots,v_m,v$ are linearly independent, then also the hyperplanes defined by $v_1^\perp,\dots,v_m^\perp,v^\perp$ are independent, thus their intersection has dimension $m$, a contradiction. So $v$ cannot be isotropic.
\end{proof}
\begin{proposition}\label{StepofStabilization}
    Let $m=\left\lfloor \frac{n}{2}\right\rfloor$. Denote $g_{s}$ the generic rank of $S^d\CC^s$, then $\DL_r$ stabilizes at $$
    r=\max\{g_{{s}}+n-2s | s\in\{0,\dots,m\}  \}.
    $$
\end{proposition}
\begin{proof}
  By the proof of Theorem \ref{thm: stab}, we know the stabilization of $f=\sum_{i=1}^{r'} l_i^d +\sum_{j=1}^{t} y_j^d$, with $l_i\in \langle l_1,\dots,l_s\rangle_\CC$ isotropic, $s\leq r'$, and $y_j$ non-isotropic, happens at the generic rank $g_{s}$ of $S^d\CC[l_1,\dots,l_s]\simeq S^d\CC^s$ plus $t$. We notice by acting as in the proof of Lemma \ref{prop: maxspaniso} that if $\langle l_1,\dots,l_s,y_1,\dots,y_t\rangle_\CC^\perp=\langle l_1,\dots,l_s\rangle_\CC$, then $t=n-2s$. This means the stabilization for this configuration happens for $r=g_{s}+n-2s$. Since $0\leq s\leq m$, the result follows.
\end{proof}

{
\begin{remark}  
The generic rank $g$ of $S^d\CC^n$ has been completely described by \cite{AH} and it is$$
g=\left\lceil \frac{\binom{n+d-1}{d}}{n}\right\rceil,
$$
except in the following cases \begin{itemize}
    \item $d=2$, where $g=n$.
    \item $3\leq n\leq 5$, $d=4$, where $g=\binom{n+1}{2}$.
    \item $(n,d)=(5,3)$, where $g=8$.
\end{itemize}
\end{remark}} 

\begin{table}[htbp]
    \centering
    \begin{tabular}{|c|c|c|c|c|c|c|c|c|c|c|c|c|c|c|c|c|c|c|}
\hline
\backslashbox{$n$}{$d$}&3&4&5&6&7&8&9&10&11&12&13&14&15\\ \hline
4&4&4&4&4&4&5&5&6&6&7&7&8&8\\
5&5&5&5&5&5&6&6&7&7&8&8&9&9\\
6&6&6&7&10&12&15&19&22&26&31&35&40&46\\
7&7&7&8&11&13&16&20&23&27&32&36&41&47\\
8&8&10&14&21&30&42&55&72&91&114&140&170&204\\
9&9&11&15&22&31&43&56&73&92&115&141&171&205\\
10&10&15&26&42&66&99&143&201&273&364&476&612&776\\\hline
    \end{tabular}
    \caption{Stabilization step of $\DL_r$ in $S^d\CC^n$}
    \label{tab:my_label}
\end{table}

\section{Examples of $\DL_N$ and conclusion}\label{Exa}

In this section, we will present several examples to show that $\DL_N$ can be sometimes the entire ambient space or it can be something smaller. We have seen in the previous section that $\DL_N$ can be equal to the weakly orthogonally decomposable tensors, now we will give an example of it being something bigger than that.
In our first example, we will see that $\DL_N$ can be the entire ambient space.
\begin{example}[Matrices]
It is classical that, for all matrices subtracting a rank-one approximation from the given matrix the result is of a lower rank. Let $M$ be the variety of $n\times n$ matrices and let $M^{\leq r}$ be the subvariety of matrices of rank less than or equal to $r$. Then $M^{\leq 1}$ is the variety of rank-one matrices and $M^{\leq n-1}$ is its dual. Now by \cite[Chapter $1$, Prop. $4.11$ and Lemma $4.12$]{GKZ} we know that the conormal variety of  $M^{\leq r}$ is the closure of the set of pairs
\[
\{(A,B)\in M\times M,\text{ s.t. }A\in M^{\leq r}\text{ and }B\in M^{\leq n-r}\}.
\]
So this way by Proposition~\ref{StructDL2}  we get that $\DL_2$ is the closure of
\[\{A+B\mid A\in M^{\leq 1}\text{ and }B\in \left(M^{\leq n-1}\cap M^{\leq 1}\right)\}=\{A+B| A\in M^{\leq 1}\text{ and }B\in M^{\leq 1}\}.\]
So $\DL_2=M^{\leq 2}$. Now we claim that $\DL_r=M^{\leq r}$. Indeed if by induction we suppose that $\DL_{r-1}=M^{\leq r-1}$, then $\DL_r$ is the closure of the $\pi_3$ projection of \[\mathcal{E}_{M^{n-1},M^{\leq 1}}\cap \left(M^{\leq r-1}\times M\times M\right).\]
We also have that $\mathcal{E}_{M^{\leq n-1},M^{\leq 1}}=\{(A,B,A+B)\mid A\in M^{n-1},B\in M^{\leq 1}\}\}$. So $\DL_r$ is the closure of the $\pi_3$ projection of
\[\{(A,B,A+B)\mid A\in M^{\leq r-1},B\in M^{\leq 1}\}\},\] and hence $\DL_r$ is the closure of $M^{\leq r-1}+M^{\leq 1}=M^{\leq r}$.
So finally we have that
\[\DL_1=M^{\leq 1}\subset \DL_2=M^{\leq 2}\subset\ldots\subset\DL_{n-1}=M^{\leq n-1}\subset \DL_{n}=M^{\leq n}=M.\]
\end{example}

In our next example, we will see that $\DL_N$ can be more than the orthogonally decomposable tensors.

\begin{example}[Regular $2\times 2\times 2$ tensors]\label{RegualEx}
Let $V=\mathbb{C}^2\otimes\mathbb{C}^2\otimes\mathbb{C}^2$. Let $X$ be the cone over the variety of regular rank-one tensors in $V$. This variety is defined by the $2\times 2$ minors of all the flattenings. Its dual $X^*$, Cayley's hyperdeterminant, is of codimension one, of degree four, and generated by the polynomial
\[y_4^2y_5^2-2y_3y_4y_5y_6+y_3^2y_6^2-\ldots-2y_1y_3y_6y_8-2y_1y_2y_7y_8+y_1^2y_8^2.\]
After running the computations for $\DL_2$ {(a Macaulay~$2$ code for computing data loci can be found in \cite[Example $4.2$]{HR22})} we find out that it is a codimenion $2$, degree $12$ variety with $3$ components. Observe that $\DL_2$ has as many components as $X^*_{sing}$ and the precise description of this singular locus can be found in \cite[Theorem $0.3$]{WZ}). The first component $\DL_{21}$ is generated by
\[\{-x_2x_5+x_1x_6-x_4x_7+x_3x_8,-x_2x_3+x_1x_4-x_6x_7+x_5x_8\},\]
the second component $\DL_{22}$ is generated by
\[\{-x_2x_5+x_1x_6-x_4x_7+x_3x_8, -x_3x_5-x_4x_6+x_1x_7+x_2x_8\}\]
and the third component $\DL_{23}$ is generated by
\[\{-x_2x_3+x_1x_4-x_6x_7+x_5x_8, -x_3x_5-x_4x_6+x_1x_7+x_2x_8\}.\]
So we can see that each component is generated by two out of the three generating polynomials for (weakly) orthogonally decomposable tensors moreover we have that (weakly) orthogonally decomposable tensors are exactly $\DL_{21}\cap\DL_{22}\cap\DL_{23}$.

Now if we continue to compute $\DL_3$ we get that it is of codimension $1$ and of degree $24$ with $3$ components. So we don't have yet stabilization of the $\DL_r's$.
Now, unfortunately, due to the lack of necessary computational power, we don't know what $\DL_4$ will be. But because the maximal rank in $V$ is $3$ and by \cite{SC10} we know that the limit $\DL_N\neq X^*$, we conjecture that $\DL_3=\DL_4$.
\end{example}
\begin{remark}
We have seen in the previous example that $\DL_2$ has as many components as $X^*_{sing}$. And that this number of components carries through for the rest of the $\DL_r$'s. This shows, in harmony with Remark~\ref{NodalSing}, that indeed there is a strong connection between our variety $\DL_N$ and the (nodal) singularities of the hyperdeterminant $X^*$. This research direction should be further persuaded.
\end{remark}
\begin{remark}
	It also remains open in the general case if and when the $\DL_r$'s stabilize to describe exactly after which step the $\DL_r$'s stabilize in the function of the given tensor format and in the function of the generic rank or the maximal rank. 	
\end{remark}

\noindent
{\bf Acknowledgements.} The authors are grateful to Jan Draisma for the stimulating discussions on the topic of this paper. We are also very thankful to all participants and organizers of the Oberwolfach Seminar on Metric Algebraic Geometry in May $2023$ who shared their ideas regarding this project. We would like to thank the anonymous referees for their valuable suggestions for improving this work.

Turatti (ORCID 0000-0003-4953-3994) has been supported by Troms{\o} Research Foundation grant agreement 17matteCR.

Horobe\c{t} was supported by the Project “Singularities and Applications” - CF 132/\ 31.07.2023 funded by the European Union - NextGenerationEU - through Romania’s National Recovery and Resilience Plan.

	\vspace{1cm}
\footnotesize {\bf Authors' address:}

\smallskip

\noindent Emil Horobe\c{t}, Sapientia Hungarian University of Transylvania \ 
\hfill {\tt horobetemil@ms.sapientia.ro}
\noindent Ettore Teixeira Turatti, UiT The Arctic University of Norway \ 
\hfill {\tt ettore.t.turatti@uit.no}
\end{document}